\documentclass[reqno,11pt]{amsart}
\usepackage{amsmath, latexsym, amsfonts, amssymb, amsthm, amscd}
\usepackage{graphics,epsf,psfrag}
\numberwithin{equation}{section}

\newtheorem{theorem}{Theorem}[section]
\newtheorem{lemma}[theorem]{Lemma}
\newtheorem{proposition}[theorem]{Proposition}

\newtheorem{remark}[theorem]{Remark}

\newcommand{\ind}{\mathbf{1}}

\newcommand{\cT}{{\ensuremath{\mathcal T}} }

\newcommand{\bP}{{\ensuremath{\mathbf P}} }

\newcommand{\bx}{{\ensuremath{\mathbf x}} }

\newcommand{\by}{{\ensuremath{\mathbf y}} }

\DeclareMathSymbol{\leqslant}{\mathalpha}{AMSa}{"36} 
\DeclareMathSymbol{\geqslant}{\mathalpha}{AMSa}{"3E} 
\DeclareMathSymbol{\eset}{\mathalpha}{AMSb}{"3F}     
\newcommand{\dd}{\,\text{\rm d}}             
 
\newcommand{\bbP}{{\ensuremath{\mathbb P}} }
\newcommand{\bbQ}{{\ensuremath{\mathbb Q}} }

\newcommand{\tm}{{\ensuremath{t^n_{\mathrm {mix}}}}}
\newcommand{\Tm}{{\ensuremath{T^n_{\mathrm {mix}}}}}
\newcommand{\gap}{{\ensuremath{\mathrm gap}} }

\newcommand{\gep}{\varepsilon}       

\newcommand{\gO}{\Omega}

\makeatletter
\def\captionfont@{\footnotesize}
\def\captionheadfont@{\scshape}

\long\def\@makecaption#1#2{%
  \vspace{2mm}
  \setbox\@tempboxa\vbox{\color@setgroup
    \advance\hsize-6pc\noindent
    \captionfont@\captionheadfont@#1\@xp\@ifnotempty\@xp
        {\@cdr#2\@nil}{.\captionfont@\upshape\enspace#2}%
    \unskip\kern-6pc\par
    \global\setbox\@ne\lastbox\color@endgroup}%
  \ifhbox\@ne 
    \setbox\@ne\hbox{\unhbox\@ne\unskip\unskip\unpenalty\unkern}%
  \fi
  \ifdim\wd\@tempboxa=\z@ 
    \setbox\@ne\hbox to\columnwidth{\hss\kern-6pc\box\@ne\hss}%
  \else 
    \setbox\@ne\vbox{\unvbox\@tempboxa\parskip\z@skip
        \noindent\unhbox\@ne\advance\hsize-6pc\par}%
\fi
  \ifnum\@tempcnta<64 
    \addvspace\abovecaptionskip
    \moveright 3pc\box\@ne
  \else 
    \moveright 3pc\box\@ne
    \nobreak
    \vskip\belowcaptionskip
  \fi
\relax
}
\makeatother
\def\writefig#1 #2 #3 {\rlap{\kern #1 truecm
\raise #2 truecm \hbox{#3}}}

\title{A product chain without cutoff}

\address{IMPA - Instituto Nacional de Matem\'atica Pura e Aplicada,
Estrada Dona Castorina 110,
Rio de Janeiro / Brasil 22460-320}
\email{lacoin@impa.br}
\author{Hubert Lacoin}

\begin{document}

\maketitle

\begin{abstract}
In this note, we construct an example of a sequence of $n$-fold product chains which does display cutoff neither for the  total-variation distance nor for the separation distance.
In addition we show that this type of product chains necessarily displays pre-cutoff.

\end{abstract}

\section{Introduction}

Consider a sequence of reversible irreducible continuous Markov chains $X^n=(X^n(t))_{t\ge 0}$, each being defined on a finite state spaces $(\gO_n)_{n\ge 0}$.
Let $\pi_n$ denote the unique reversible probability measure associated to $X^n$. 
It is a classic result of Markov chain theory that for any initial condition the distribution 
of $X^n(t)$ converges to $\pi_n$ when $t$ goes to infinity.
We let $P^n_t$ denote the Markov semigroup associated to $X^n$ and $d_n(t)$ resp.\ 
$d_n^s(t)$ denote the distance to equilibrium for the total variation 
distance and separation distance (they are defined by taking the maximal distance over all initial condition)

\begin{equation}\label{septv}
\begin{split}
d_n(t)&:=\max_{x\in \gO_n} \| P^{n}_t(x,\cdot)-\pi_n \|_{TV},\\
d^s_n(t)&:= 1- \min_{x,y \in \gO_n} \frac{P^n_t(x,y)}{\pi_n(y)}.
\end{split}
\end{equation}

When we have to consider only one Markov chain $X$ in Section \ref{PTV}, we will use the same notation without $n$.

The sequence $X^n$ is said to display \textsl{cutoff} if $d_n(t)$ drops abruptly from $1$ to $0$ on the appropriate time scale.
More precisely, if one defines the mixing 
time corresponding to the distance $a\in(0,1)$ to be
\begin{equation}
\tm(a):=\inf\{ t \ | \ d_n(t)< a \}. 
\end{equation}
the chain is said to display \textsl{cutoff} if for any $\gep\in(0,1/2]$
\begin{equation}
\lim_{n\to \infty} \tm(\gep)/\tm(1-\gep)=1.
\end{equation}
We follow the definition given in \cite[pp .248]{cf:LPW} and say displays pre-cutoff if
\begin{equation}
\limsup_{\gep \to 0+}\limsup_{n\to \infty} \tm(\gep)/\tm(1-\gep)<\infty.
\end{equation}
Note that one can replace $\tm$ by $t_s^n$ the mixing time for the separation distance.

\medskip

The term cutoff was coined by Aldous and Diaconis \cite{cf:AD} and its occurrence  for the transposition shuffle was proved by 
Diaconis and Shahshahani \cite{cf:DS}. It is thought to hold for many natural sequences of Markov chain as soon as 
\begin{equation}\label{prodcond}
\tm(1/4) \times \gap_n=\infty \tag{H}
\end{equation} 
where $\gap_n$ corresponds to the spectral gap of the chain $X^n$  
(see e.g. \cite[Chapter 12 and Chapter 18]{cf:LPW} for the definition of the spectral gap and an account on the cutoff phenomenon).
More precisely the conditioncis necessary and it is was proposed by Peres as a natural sufficient 
condition provided the chain is ``nice enough". As \eqref{prodcond} is in fact known to be a necessary condition for pre-cutoff, 
this would imply in particular than pre-cutoff implies cutoff for ``nice chains''.

\medskip

Shortly after \eqref{prodcond} was proposed as a sufficient condition for cutoff, Aldous constructed a chain that satisfies 
\eqref{prodcond} and displays pre-cutoff, but for which cutoff does not hold. 
Pak also constructed a counter-example (with no pre-cutoff) 
which is a random walk on a Cayley graph (see \cite[pp 253--256]{cf:LPW}). 
Since then it has been a challenge to find a large class of Markov chain for which the \eqref{prodcond} condition is a sufficient one.
Note that Chen and Saloff-Coste have shown that \eqref{prodcond} is a sufficient condition in full generality 
when distance to equilibrium is measured by the $L^p$ norm \cite{cf:CSC}.
Let us note also that \cite[Proposition 7]{cf:LS} establishes that cutoff holds 
for large product chains provided one has a good-control on the supremum norm of the relative density of the marginals.

\medskip

We define $Y^n$ the chain corresponding to $n$ independent copies of $X^n$ (its $n$-th power)

\begin{equation}
Y^n(t):=(X^{n}_1(t),\dots,X^n_n(t)).
\end{equation}

In this note we show that the sequence $Y^n$ always displays pre-cutoff, and we construct construct a sequence of chain $X$ which is such that $Y$ displays no cutoff (whereas $X$ does), showing that condition \eqref{prodcond} is not a sufficient condition for cutoff for chains that are large powers of a simpler one.

\medskip

\section{Pre-cutoff for product chains}

We let $D_n$, $D_n^s$, $Q^n_t$, $\Tm$ and $T^n_s$, and $\mu_n:= \pi_n^{\otimes n}$  denote the distances to equilibrium, semigroup, and mixing time and equilibrium measure for the chain $Y^n$.
We have

\begin{proposition}\label{preco}
For any sequence of non-trivial Markov chain $X^n$ one has
\begin{equation}\label{crrr}
\limsup_{n\to \infty} \frac{\Tm(1-\gep)}{\Tm(\gep)}\le 2.
\end{equation}
The result also holds when the total-variation distance is replaced by the separation distance.
\end{proposition}

\begin{remark}
 In the first draft of this paper, the optimal bound of $2$ for the mixing time ratio was proved to hold only for the 
 separation distance. The idea of using the Hellinger distance  to obtain an optimal bound  also 
 for the total-variation distance (developped in Section \ref{PTV}) is due to Yuval Peres.
\end{remark}

\subsection{Proof of Proposition \ref{preco} for the separation distance}

The separation distance to equilibrium for $Y^n$ is given by

\begin{equation}\label{dd}
D^s_n(t):= 1- \min_{\bx, \by \in \gO_n} \frac{Q^n_t(\bx,\by)}{\mu_n(\by)}= 1- (1- d^s_n(t))^n.
\end{equation}
Hence for $\gep$ fixed and $n$ sufficiently large we have

\begin{equation}\label{technios}
t^n_s(n^{-2/3})\le T^n_s( 1-\gep )\le T^n_s( \gep )\le t^n_s(n^{-4/3})\le 2 t^n_s(n^{-2/3}),
\end{equation}
where the last inequality is due to the submultiplicativity property for the separation distance 

\begin{equation}\label{submul}
d_s(a+b)\le d_s(a)d_s(b).
\end{equation}
Hence the result.

\qed

For the total-variation distance,  can obtain \eqref{crrr} with $4$ instead of $2$ on the r.h.s.
simply by using the following comparison 
between the total variation distance and separation distance
for reversible Markov chains initially proved in \cite{cf:AD2} (see also \cite[Lemma 6.13 and Lemma 19.3]{cfLPW})
$$d_n(t) \le d_n^s(t)\le 4 d_n(t/2)).$$

\subsection{Proof of Proposition \ref{preco} for the total-variation distance}\label{PTV}

For an optimal result, we need to use the \textit{Hellinger distance} which has the property of behaving nicely for
product. This section starts with the introduction of notation and recalling some classical inequalities.

\medskip

Given $\mu$ and $\nu$ two probability measures on a common finite state space $\gO$, tj
\begin{equation}
 d^H(\mu,\nu):= \sqrt{\sum_{x\in \gO}\left(\sqrt{\nu(x)}-\sqrt{\mu(y)}\right)^2}.
\end{equation}
We have the following comparisons with the total-variation distance (see for instance \cite[(20.22) and (20.29)]{cf:LPW})
\begin{equation}\label{comparison}
 \| \mu-\nu \|_{TV} \le d^H(\mu,\nu)\le \sqrt{2 \| \mu-\nu \|_{TV}}
\end{equation}
We set 
\begin{equation}\label{def:dhn}
 d^H_n(t):= \sup_{x\in \gO} d^H(P^{n}_t(x,\cdot),\pi_n),
\end{equation}
and let $D^H_n(t)$ denote the counterpart of $d^H_n$ for the chain $Y_n$.
Similarly to \eqref{dd}, it is easy to remark (see also \cite[Exercice 20.5]{cf:LPW}) that 
\begin{equation}
1- \frac{1}{2}\left(D^H_n(t)\right)^2= \left (1- \frac 1 2 (d^H_n(t))^2\right)^n.
\end{equation}
Hence we know that $D^H_n(t)$ is close to $\sqrt{2}$ resp. $0$ (and hence by \eqref{comparison}
that $D_n(t)$ is close to $1$ resp. $0$) if and only if $\sqrt{n}d^H_n(t)$ is close to infinity resp. $0$.
What we need to conclude is that there is a time window $[t_n, 2t_n]$ for which the 
Hellinger distance drops from $n^{-1/2+\delta}$ to $n^{-1/2-\delta}$.
We achieve this by proving the following property of the Hellinger distance for reversible Markov chains
\begin{lemma}\label{hellinger}
For any reversible irreducible Markov chain and any $t\ge 0$,
\begin{equation}
d^H(2t)\le 7(d^H(t))^{5/4}.
\end{equation}
\end{lemma}
With this results at hand, it is easy to prove, that similarly to \eqref{technios}, for 
$$t_n:=\inf\{ t \ | \  d^H_n(t)\le n^{-3/7} \}.$$
one has for any $\gep \in (0,1/2)$, for all $n$ sufficiently large 

\begin{equation}
t_n\le  \Tm(1-\gep) \le \Tm(\gep)\le 2 t_n.
\end{equation}

%

\begin{proof}[Proof of Lemma \ref{hellinger}]

We  introduce now $\bar d(t)$ defined as 
\begin{equation}
\bar d(t):=\max_{x,y\in \gO^2} \| P_t(x,\cdot)-P_t(y,\cdot) \|_{TV}.
\end{equation}
Note that, as the chain is assumed to be reversible $\bar d (t)$ also 
correspond to the operator norm for $P_t$ acting on integrable functions with mean $0$, or more precisely
\begin{equation}\label{normop}
 \bar d(t) =\max_{\{f \in l_1(\pi) \ |  \ \pi(f)=0\}} \frac{\| P_t f \|_{l_1(\pi)} }{\| f\|_{l_1(\pi)}},
\end{equation}
where 
$$P_t f(x):=\sum_{y\in \gO} P_t(x,y)f(y).$$

The function $\bar d(t)$ compares well with $d(t)$ and is submultiplicative 
(see for instance \cite[Chapter 4]{cf:LPW})
\begin{equation}\label{propbard}\begin{split}
 d(t)\le \bar d(t)\le 2 d(t) \\
 \bar d(t+s)\le \bar d(t)\bar d(s).
\end{split}\end{equation}
Combining \eqref{comparison} and \eqref{propbard}, we have for every $t$
\begin{equation}\label{comparison2}
 \bar d(t)/2 \le d^H(t)\le \sqrt{2 \bar d(t)}.
\end{equation}

\medskip

Let us try now to prove the result from
 \eqref{comparison2} (inequality on the left) and \eqref{propbard} in a naive way.
 We have
\begin{equation}\label{cromox}
\bar d(2t)\le (\bar d(t))^2 \le 4 (d^H(t))^2,
\end{equation}
and hence using \eqref{comparison2} again (inequality on the right) we obtain
\begin{equation}
 d^H(2t)\le  \sqrt{8} d^H(t),
\end{equation}
which is not satisfying.

\medskip

To find a way out, we have to prove that if the inequality on the left in
\eqref{comparison2} is sharp for $t$, the inequality on the right cannot be sharp for $2t$.

\medskip

We set $u:=d^H(t)$ (note that we can assume $u\le 1$ as the result is trivial for $u\ge 1$)
Let $x$ an  element of $\gO$ for which  $d^H(2t)= d^H(P_t(x,  \ \cdot \ ), \pi)$.
Let $g$ denote the density of $P_{t}(x,\cdot)$ with respect to $\pi$
and $g'$ denote the density of $P_{2t}(x,\cdot)$ with respect to $\pi$.

\medskip

We have from our definitions
\begin{equation}\label{groki}\begin{split}
 \sqrt{\int \left( \sqrt {g'(y)}-1 \right)^2 \pi( \dd y)}&=(d_H(2t))^2,\\
 \sqrt{\int \left(\sqrt {g(y)}-1 \right)^2 \pi( \dd y)}&\le u^2.
\end{split}\end{equation}

Our first step is the contribution to the total variation 
distance $\| P_t(x,  \ \cdot \ ), \pi \|$  of the set $\{y  \ | \ |g(y)-1|\ge u^{1/2}\}$ is much smaller than $u$.

\begin{lemma}
We have for all $u\le 1$
\begin{equation}
\int |g(y)-1| \ind_{\{|g(y)-1|\ge u^{1/2}\}}\dd \pi(\dd y) \le 10 u^{3/2}.
\end{equation}
\end{lemma}

\begin{proof}
 We have to show that
 \begin{multline} \label{mush}
 \int |g-1| \ind_{\{|g(y)-1|\ge u^{1/2} \}}\dd \pi(\dd y)\\
 \le  10 \int \left(\sqrt{ g(y)}-1\right)^2 u^{-1/2} \ind_{\{|g(y)-1|\ge u^{1/2}\}}\dd \pi(\dd y),
 \end{multline}
 and we conclude by using \eqref{groki}.
 The inequality \eqref{mush} is obtained by noticing that when $g\ge 2$ we have 
 \begin{equation}
 |g-1|\le (3-2\sqrt{2}) |\sqrt g-1|^2,
 \end{equation}
 while when $g\in (0,2)$, $|g-1|\ge u^{1/2}$,
 we have 
 \begin{equation}
 |g-1|\le u^{-1/2} |g-1|^2\le \frac {u^{-1/2}} {(\sqrt 2-1)^2}|\sqrt g-1|^2.
 \end{equation}
\end{proof}
Now
we can decompose $g-1$ into a sum of two function $h_1$ and $h_2$: one which has a small $l_\infty$ norm, 
and one which has a small $l_1$ norm.

\begin{equation}\begin{split}
 h_1(y):= (g-1)(y) \ind_{\{|g(y)-1|< u^{1/2}\}},\\
 h_2(y):= (g-1)(y) \ind_{\{|g(y)-1|\ge u^{1/2}\}}.
\end{split}\end{equation}

We have
\begin{equation}\begin{split}
 \|h_1\|_{l_\infty}&\le  u^{1/2},\\
 \|h_2\|_{l_1(\pi)}&\le  10 u^{3/2}.
\end{split}\end{equation}
Setting $h'_i:=P_{t}h_i$ one has

\begin{equation}
 g'-1=h'_1+h'_2.
\end{equation}
From \eqref{normop} one has (using \eqref{comparison2} to bound $\bar d(t)$)
\begin{equation}\label{lesborns}
\begin{split}
 \|h'_2\|_{l_1(\pi)}&\le \bar d(t)\|h_2\|_{l_1(\pi)}\le 20 u^{5/2},\\
  \|h'_1\|_{l_\infty}&\le \|h_1\|_{l_\infty}\le u^{1/2}.
  \end{split}
 \end{equation}
Moreover 
\begin{equation}\label{lesborns2}
 \|g'-1\|_{l_1(\pi)}\le d(2t)\le  \bar d(t)^2\le 4u^2.
 \end{equation}
We are now ready to bound $(d^H(2t))^2$.
 We split it into two parts. The first one is bounded thanks to \eqref{lesborns2}
\begin{equation}
 \int \left(\sqrt{g'(y)}-1\right)^2  \ind_{\{|g'(y)-1|\le 2u^{1/2}\}}\pi(\dd y) 
 \le 2u^{1/2} \int |g'(y)-1| \pi(\dd y)\le 8 u^{5/2}.
 \end{equation}
 For the second part, note that as $(\sqrt{g'}-1)^2\le |g'-1|$ we have
\begin{multline}
 \int \left(\sqrt{g'(y)}-1\right)^2 \ind_{|g'(y)-1|\ge 2u^{1/2}}\pi(\dd y)\\
 \le 
  \int |g'(y)-1| \ind_{|h'_2(y)|\ge u^{1/2}}\pi(\dd y)\\
    \le \int 2|h'_2(y)| \ind_{|h'_2(y)|\ge  u^{1/2}}\pi(\dd y)\le 40 u^{5/2}.
\end{multline}
where the last inequality comes from \eqref{lesborns}, and the one before from the fact that 
$$|g'-1|\le |h'_1+h'_2|\le |h'_2|+u^{1/2}.$$
This allows us to conclude.

\end{proof}

\section{An example without cutoff}

\subsection{Construction}

Let us now define a sequence $X^n$ such that $Y^n$ displays no cutoff. The idea build on the counter example 
of Aldous displayed on \cite[pp 256]{cf:LPW}.
The state-space of $X^n$ is the vertex set $V_n$ of a graph $G_n$ with $2n+1$ edges and  $2n+1$  vertices defined as follows:
\begin{itemize}
\item There is a segment of $2n$ edges linking $2n+1$ vertices. We call $A$ and $C$ its ends.
\item There is an extra edge linking the midle point of the segment (which we call $B$) to $C$.
\end{itemize}
The transition rates are positive on the edges of $G_n$ and are specified in the caption of Figure \ref{graf} .

\begin{figure}[hlt]
 \epsfxsize =14.0 cm
 \begin{center}
  \psfrag{en}{$\gep_n$}
  \psfrag{een}{$\gep'_n$}
    \psfrag{A}{$A$}
        \psfrag{B}{$B$}
           \psfrag{C}{$C$}
       \psfrag{n}{$n$}
           \psfrag{1}{$1$}
        \psfrag{1/n}{$1/n$}
            \psfrag{1-1/n}{$1-1/n$}
                     \epsfbox{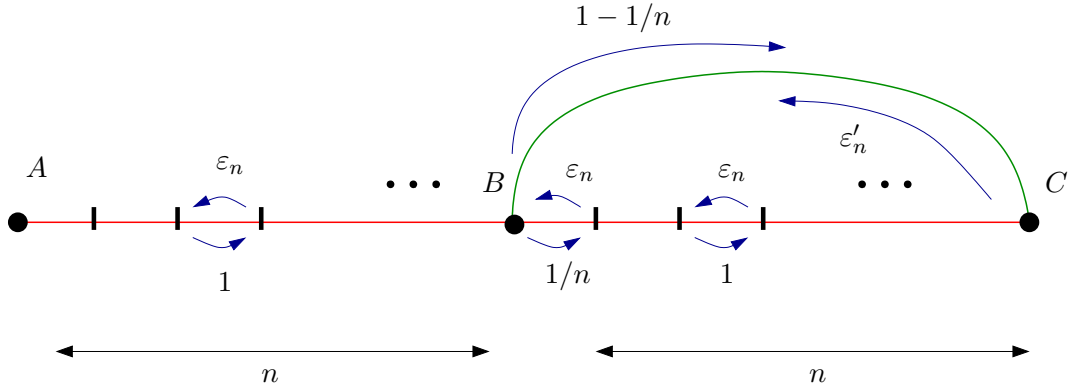}
 \end{center}
 \caption{\label{graf}  
The graph $G_n$ together with the transition rates of $X^n$: the two segments that are represented in red are of length $n$.
The jump rate are represented in blue above the arrows. In the direction from $A$ to $C$  
the jump rate is always one except at point $B$
where the jump rate to $C$ (along the green edge) is equal to $1-1/n$ while the probability to jump towards 
$C$ on the red path is $1/n$.
The jumps in the direction of $A$ along red edges are equal to $\gep_n:=2^{-n^2}$.
The jump rate from $C$ to $B$ along the green edge is equal to $$\gep'_n:=\frac{(n-1)2^{n^3}}{n^2}.$$}
 \end{figure}
With this definition it is not difficult to check  that $X^n$ is a reversible Markov chain. 
We have chosen $\gep_n$ to be exponentially small 
 but the result we are going to present whould remain valid for $\gep_n=1/2$ for all $n$ (or any other value smaller than $1$). Note that 
 the value of $\gep'_n$
is determined by that of $\gep_n$ in order to have reversibility.

\begin{proposition}\label{cutoff}
The construction above satisfies the following property
\begin{itemize}
\item[(i)] The sequence $X^n$ displays cutoff around time $n$, both in separation and total-variation distance.
\item[(ii)] The sequence does not $Y^n$ display cutoff as 
\begin{equation}
\Tm(a)=\begin{cases} 2n(1+o(1))  &\quad \text{ for }   a\in (1-e^{-1},1),\\
n(1+o(1))&\quad \text{ for }   a\in (0,1-e^{-1}).
\end{cases}
\end{equation}
\end{itemize}
(the notation means that for a fixed $a\ne (1-e^{-1})$, $\Tm(a)/n$ converges either to $1$ or $2$.)
The same holds for the separation distance.
\end{proposition}

\begin{remark}
 The above Proposition shows that the inequality \eqref{crrr} concerning the ratio of the mixing time is optimal.
 \end{remark}

The main idea of the proof is that the total variation distance can be expressed 
in terms of the distribution of the time $\tau$ or $\cT$ needed to reach 
$A$ (for $X^n$) or ${\bf C}:=(C,C,\dots,C)$ (for $Y^n$) starting from $A$. 
In particular, there is cutoff if and only if this time is concentrated around its mean. 
For $X^n$ we show that $\tau$ concentrates around $n$, whereas for $Y^n$, $\cT$  will be about $2n$ 
if at least one of the coordinates $X^{n}_1$ decides to use the red path between $B$ and $C$ 
(which happens with a non-vanishing probability).

\subsection{Proof of Proposition \ref{cutoff}}

The equilibrium measure $\pi_n$ gives a weight $1-O(2^{-n^2})$ to the vertex $C$, 
and hence the equilibrium measure $\mu_n$ of $Y^n$, gives weight  $1-O(n 2^{-n^2})$ to ${\bf C}:=(C,C,\dots,C)$.
Because of this remark  we have
\begin{equation}\label{remark}
d_n(t)=1-\min_{x\in V_n} P_t(x,D)+o(1)  \quad  \text{ and } \quad  D_n(t)=1-\min_{\bx\in V^n_n} P_t(\bx,{\bf D})+o(1).
\end{equation}

For $x\in V_n$ or ${\bf x}\in V_n^n$, let $\bbP^{n,x}$ resp.\ $\bbQ^{n,\bx}$ be the law of $X^n(t)$ starting from $\bx$ resp. the law of $Y^n(t)$ and let 
$\tau$, resp. $\cT$ be the first hitting time of $D$ resp. ${\bf D}$.

\begin{lemma}\label{crooom}

We have 
\begin{equation}\begin{split}
d_n(t)&= \bbP^{n, A}(\tau>t)+o(1), \\
D_n(t)&= \bbQ^{n, {\bf A}} (\cT>t)+o(1),
\end{split}\end{equation}
meaning that 
\begin{equation}\begin{split}
\lim_{n\to \infty} \sup_{t\ge 0} | d_n(t)-\bbP^{n, A}(\tau>t)|&=0, \\
\lim_{n\to \infty} \sup_{t\ge 0}| D_n(t)-\bbQ^{n, \bf A} (\cT>t) |&=0.
\end{split}\end{equation}

\end{lemma}

\begin{proof}
We provide the proof for $d_n(t)$ as the other is identical.
First let us prove the result for $t<3n$, and we will check later that for  $t>3n$ both $d_n(t)$ and $\bbP^n(\tau>t)$ are $o(1)$.
Now the probability that a jump in the direction $A$ (a backtrack) 
occurs before time $3n$ is exponentiallty small in $n$ and thus from \eqref{remark} we have

\begin{equation}
P_t(x,D)=\bbP^{n,x}(\tau\le t)+o(1) 
\end{equation}
Hence from \eqref{remark}, it is sufficient to check that the minimum of $\bbP^{n,x}(\tau \le t)$ is reached for $A$ (up to some $o(1)$ correction).

\medskip

From an obvious coupling , we see that $A$ is the point of the segment $AB$ which makes $\tau$ the largest. 
It remains to check that starting from one of the $n-1$ inside points the red segment  $B$ and $C$ cannot make $\tau$ larger:
by conditioning to the event that $X^n$ does not backtrack before $t$ (which is an event of almost full probability)
we see that $\tau$ starting from $A$ is bounded from below by a sum of $n+1$ IID standard exponentials whereas 
in the $BC$ branches it is bounded from above by the sum of 
$n$ IID standard exponentials.

\medskip

Finally, for $t=3n$, as conditioned on no backtrack, $\tau$ starting from $A$ is a bounded from above by a sum of $2n$ 
IID standard exponentials, both $\bbP^n(\tau>3n)$ and 
 $d_n(t)$ are $o(1)$ (and the fact both functions are decreasing allows to conclude for larger values of $t$).

\end{proof}

From Lemma \eqref{crooom} one has
\begin{equation}
D_n(t)=1- \left[\bbP^{n,A}(\tau\le t)\right]^n+o(1).
\end{equation}
Hence $D_n(t)$ is in a neighborhood of $1$  resp. $0$ if and only  if  $n \bP^{n,A}(\tau>t)$ is in a neighborhood of infinity resp. $0$.

\medskip

Concerning $X^n$, one can remark that conditioning to the event that $X^n$ does not backtrack and uses a short branch to reach $D$, $\tau$ is a sum of $3n+1$ IID standard exponentials. Hence as the event to which we are conditioning has a probability tending to one, we have

\begin{equation}
\lim_{n\to \infty} d_n( ns )= \begin{cases} 1  \quad &\text{ if } s<1, \\
  0 \quad &\text{ if } s>1. \end{cases}
\end{equation}
and $X^n$ exhibits cutoff. 
However, the slow branch plays a crucial role for the product chain as the probability to hit $D$ from the longer branch asymptotically behaves like $n^{-1}$.
As a consequence we have 

\begin{lemma}

\begin{equation}
\lim_{n\to \infty} n \bbP^{n,A}(\tau > ns)= \begin{cases} \infty   \quad &\text{ if } s<1, \\
  1 \quad &\text{ if } s\in (1,2),  \\
  0 \quad &\text{ if } s>2.
  \end{cases}
\end{equation}

\end{lemma}

\begin{proof}
Under $\bbP^{n,A}$ the probability that $X^n$ backtrack before time $3n$ is exponentially small in $n$ and thus can be neglected.
Conditioned on no backtracking, the probability to use the red segment $BC$ is equal to $n^{-1}$.
Now conditioned on using the red segment, $\tau$ is a sum of $2n$ IID standard exponentials 
whereas conditioned on using the green edge
$\tau$ is a sum of $n+1$ IID standard exponentials. Hence the result.
\end{proof}

This implies

\begin{equation}
\lim_{n\to \infty} D_n( ns )= \begin{cases} 1   \quad &\text{ if } s<1, \\
  1-e^{-1} \quad &\text{ if } s\in(1, 2),  \\
  0 \quad &\text{ if } s>2.
  \end{cases}
\end{equation}
and hence $Y^n$ exhibits no cutoff for total variation distance.

\medskip

Now let us show that cutoff also holds for the separation distance.
This amounts essentially to prove the following 

\begin{lemma}
For all $n$ sufficiently large,
for any $x,y\in V_n \setminus \{ C\}$, for all $n/2 \le t \le 3n$ one has 

\begin{equation}
P^n_t(x,y)\ge \pi_n (y)
\end{equation}

\end{lemma}

\begin{proof}
From reversibility
\begin{equation}
 \frac{P^n_t(x,y)}
 {\pi_n (y)}\ge  \frac{P^n_t(y,x)}
 {\pi_n (x)}
\end{equation}
so that one can without loss of generality consider that $x$ is the point closer to $A$ on the red segment.
Let $d$ be the number of red edges between $x$ and $y$. 
Then $P^n_t(x,y)$ is bounded from below by the probability of the event: in the time interval $[0,t]$ 
the walk $X^n$ (starting from $x$) makes exactly 
$d$ jumps following the red path from $x$ to $y$.

\medskip

As  the jump rate for $X$ is always of order $1$ (except at point $C$), 
the probability of making exactly $d$ jumps in the time interval $[0,t]$ is larger than $e^{-C_1n}$ for some constant $C_1$.
The probability of not following the red path conditioning to the number of jump is at least $1/2n$ 
(a backtrack is exponentially unlikely, and if the path goes through $B$ the chance of choosing the right direction there is 
equivalent to $n^{-1}$). Hence there exists a constant $C_2$ such that when $n$ is sufficiently large 
\begin{equation}
\forall t \in (n/2, 3n), \quad \forall x,y\in  V_n \setminus \{ C\}, P^n_t(x,y)\ge e^{-C_2 n}.
\end{equation} 

As $\pi_n(y)\le 2^{-n^2}$ for all $y \ne C$, this is sufficient to conclude.
\end{proof}

From the previous Lemma (and the definition \eqref{septv} and reversibility),  one has for all $t \in (n/2,3n)$
\begin{equation}
d^s_n(t):= 1- \min_{x \in \gO_n} \frac{P^n_t(x,C)}{\pi_n(C)},
\end{equation}
which according to \eqref{remark} shows that the difference
between total-variation and separation distance for this chain is negligible.

{\bf Acknowledgement:}
The author is grateful to Perla Sousi and Yuval Peres to have make 
him now about the question of cutoff for product chains and for enlightening discussions.
In particular the author wishes to thank Yuval Peres for suggesting a proof of \eqref{crrr} 
for the total-variation distance.

\bibliographystyle{plain}

\end{document}